\documentclass{amsart}
\usepackage[czech,english]{babel}
\usepackage{amsmath,amsthm,amssymb,amscd}
\usepackage{xypic}

\theoremstyle{plain}
\newtheorem{thm}{Theorem}[section]
\newtheorem{lem}[thm]{Lemma}
\newtheorem{prop}[thm]{Proposition}
\newtheorem{cor}[thm]{Corollary}

\theoremstyle{definition}
\newtheorem{defn}[thm]{Definition}
\newtheorem{nota}[thm]{Notation}

\theoremstyle{remark}
\newtheorem{rem}[thm]{Remark}

\theoremstyle{plain}
\newtheorem*{thmGeneric}{Theorem}
\newtheorem*{corDecVsKapl}{Corollary~\ref{cor:dec_kapl}}
\newtheorem*{propAddCl}{Proposition~\ref{prop:add_deconstr}}
\newtheorem*{thmBoundingLength}{Theorem~\ref{thm:bounding length}}
\newtheorem*{thmCpxs}{Theorem~\ref{thm:deconstr_cpxs}}

\renewcommand{\iff}{if and only if }
\newcommand{\st}{such that }

\newcommand{\la}{\longrightarrow}
\newcommand{\CGcs}{{\Cpx\G,c.s.}}

\DeclareMathOperator{\Hom}{Hom}

\DeclareMathOperator{\End}{End}
\DeclareMathOperator{\Ext}{Ext}
\newcommand{\ExtG}{{\Ext_\G^1}}
\DeclareMathOperator{\Ker}{Ker}
\DeclareMathOperator{\Img}{Im}
\DeclareMathOperator{\Coker}{Coker}

\DeclareMathOperator{\Sum}{Sum}

\DeclareMathOperator{\Subobj}{Subobj}
\newcommand{\li}{\varinjlim}

\newcommand{\Mod}[1]{\hbox{\rm Mod-}{#1}}

\newcommand{\Qco}[1]{\mathfrak{Qco}({#1})}
\newcommand{\Filt}[1]{\hbox{\rm Filt-}{#1}}
\newcommand{\dg}[1]{\hbox{\rm dg-}\tilde{#1}}
\newcommand{\Cpx}[1]{\mathbf{C}({#1})}

\newcommand{\Der}[1]{\mathbf{D}({#1})}
\newcommand{\Ab}{\mathbf{Ab}}

\newcommand{\C}{\mathcal{C}}

\newcommand{\E}{\mathcal{E}}
\newcommand{\F}{\mathcal{F}}
\newcommand{\G}{\mathcal{G}}
\newcommand{\clH}{\mathcal{H}}
\newcommand{\I}{\mathcal{I}}
\newcommand{\clL}{\mathcal{L}}

\newcommand{\Q}{\mathcal{Q}}
\newcommand{\clS}{\mathcal{S}}

\newcommand{\U}{\mathcal{U}}

\newcommand{\Z}{\mathbb{Z}}

\newcommand{\Pow}{\mathcal{P}}
\DeclareMathOperator{\lev}{lev}
\newcommand{\card}[1]{\left\lvert{#1}\right\rvert}

\title{Deconstructibility and the Hill lemma in Grothendieck categories}

\author{Jan \v{S}\v{t}ov\'\i\v{c}ek}
\address{
Charles University in Prague, Faculty of Mathematics and Physics \\
Department of Algebra \\
Sokolovska 83, 186 75 Praha 8, Czech Republic
}
\email{stovicek@karlin.mff.cuni.cz}

\subjclass[2000]{18E15 (primary), 18G35, 16D70 (secondary)}
\keywords{Grothendieck categories, deconstructible classes, Generalized Hill Lemma, chain complexes}

\date{January 26, 2011}

\thanks{%
The research was supported by the grant GA\v{C}R P201/10/P084 and the research project MSM~0021620839.
}

\begin{document}

\begin{abstract}
A full subcategory of a Grothendieck category is called deconstructible if it consists of all transfinite extensions of some set of objects. This concept provides a handy framework for structure theory and construction of approximations for subcategories of Grothendieck categories. It also allows to construct model structures and t-structures on categories of complexes over a Grothendieck category. In this paper we aim to establish fundamental results on deconstructible classes and outline how to apply these in the areas mentioned above. This is related to recent work of Gillespie, Enochs, Estrada, Guil Asensio, Murfet, Neeman, Prest, Trlifaj and others.
\end{abstract}

\maketitle

% -----------------------------------------------------------------------------
%\setcounter{tocdepth}{1}
%\tableofcontents

% ------------------------------------------------------------------------------
\section*{Introduction}

The main subject of this paper is the notion of a deconstructible class in a Grothendieck category. Roughly speaking, these are classes whose all objects can be formed by transfinite extensions from a set of objects. Such classes arise in plenitude in structure theory of modules and in homological algebra, the classes of abelian $p$-groups, projective modules or flat quasi-coherent sheaves serving as examples. We aim to establish easy to use yet strong enough properties which cover a considerable part of technicalities tackled in the recent work of Gillespie, Estrada, Guil Asensio, Murfet, Neeman, Prest, Trlifaj and others; we refer to~\cite{EGPT,G3,G2,G1,H3,Mur,Nee,Nee3}.

\smallskip

In the last decade, deconstructible classes in module categories have been implicitly used in analyzing structure of modules (e.g.~\cite{GT,AHT}), to build up theory for approximations and cotorsion pairs~\cite{GT}, which was among others successfully applied to solve the Baer splitting problem~\cite{ABH} and to prove finite type of tilting modules~\cite{BH,BS}.
As explained in~\cite[\S4]{SaSt}, it turned out recently that analogous ideas apply in a much broader setting. For example in:

\begin{enumerate}
\item Gillespie's construction~\cite{G3,G2,G1} of flat monoidal model structures on module and sheaf categories (see~\cite{H3} for a nice overview).

\item Neeman's~\cite{Nee} and Murfet's~\cite[3.16]{Mur} proof of existence of certain triangulated adjoint functors without using Brown representability (see~\cite{Nee4} for an overview and motivation).
\end{enumerate}

As we explain in a moment, a key point in both settings is the fact that certain classes of chain complexes over modules or sheaves are deconstructible. Here we get to the main aim of the paper: We show that the notion of deconstructibility in Grothendieck categories is very well-behaved. The main reason for the favorable properties is a so-called Generalized Hill Lemma (Theorem~\ref{thm:generalized hill lemma} in this paper) which originated in the theory of abelian $p$-groups~\cite{Hill}. It roughly says that given one expression of an object in a Grothendieck category as a transfinite extension with prescribed factors, there are typically many other such expressions. Though being somewhat technical, the Hill Lemma has several nice consequences.

\smallskip

First of all, if $\F$ is a deconstructible class in a Grothendieck category and $X \subseteq F$ is a ``small'' subobject of $F \in \F$, there is always a ``small'' subobject $Y$ \st $X \subseteq Y \subseteq F$ and $Y, F/Y \in \F$. Such a property was very important for instance in Gillespie's considerations in~\cite{G1}. Giving it a precise meaning and a name in Definition~\ref{def:kaplansky}, we get: 

\begin{corDecVsKapl}
Any deconstructible class in a Grothendieck category is a $\kappa$-Ka\-plan\-sky class for arbitrarily large regular cardinals $\kappa$. On the other hand any Kaplansky class which is closed under taking direct limits is deconstructible.
\end{corDecVsKapl}

Next, deconstructible classes are closed under some natural constructions:

\begin{propAddCl}
The following assertions hold for any Grothendieck category $\G$:
\begin{enumerate}
\item If $\E$ is a deconstructible class in $\G$ and $\F$ is the class consisting of all direct summands of objects from $\E$, then $\F$ is deconstructible, too.

\item The intersection of any set indexed collection of deconstructible classes is deconstructible again.
\end{enumerate}
\end{propAddCl}

\smallskip

A central requirement for homological algebra and many applications, such as those mentioned above, is existence of approximations. Given a full subcategory $\F$ of a category $\G$, we say that a morphism $f: F \to X$ is an \emph{$\F$-precover} of $X$ if $F \in \F$ and any other morphism $f': F' \to X$ with $F' \in \F$ factors through $f$. The class $\F$ is called \emph{precovering} if each $X \in \G$ admits an $\F$-precover. \emph{Preenvelopes} and \emph{preenveloping classes} are defined dually. Now we have:

\begin{thmGeneric} \emph{(\cite[2.14]{SaSt} and proof of~\cite[4.19]{SaSt})}
Let $\F$ be a deconstructible class in a Grothendieck category. Then $\F$ is precovering. If in addition $\F$ is closed under products, it is also preenveloping.
\end{thmGeneric}

\smallskip

For the results by Gillespie, Neeman and Murfet mentioned above, one needs certain classes of chain complexes of quasi-coherent sheaves to be precovering or preenveloping, which one proves via deconstructibility. To start with, many classes of modules are well-known to be deconstructible (consult~\cite{GT}).
If one wishes to deal with sheaves, good news coming from~\cite{EGPT} is that one can prove deconstructibility of a class of sheaves locally:

\begin{thmGeneric} \emph{(\cite[3.7 and (proof of) 3.8]{EGPT})} \label{thm:sheaves}
Assume that $\mathbb{X}$ is a scheme with the structure sheaf $\mathcal{O}$, and let $V$ be a family of open sets which covers $\mathbb{X}$ as well as all intersections $u \cap v$ for each $u,v \in V$. Assume further that we are given for each $v \in V$ a deconstructible class $\F_v \subseteq \Mod{\mathcal{O}(v)}$. Then the class of quasi-coherent sheaves defined as
$ \F = \{\mathcal{X} \in \Qco{\mathbb{X}} \mid \mathcal{X}(v) \in \F_v \textrm{ for each } v \in V \} $
is deconstructible in $\Qco{\mathbb{X}}$.
\end{thmGeneric}

The passage to complexes is also made straightforward by the Hill Lemma. Given a Grothendieck category $\G$ and a deconstructible class $\F \subseteq \G$, various classes of complexes constructed from $\F$ are automatically deconstructible in $\Cpx\G$. Here we use Notation~\ref{not:F tilde and dg-F}, originating in~\cite{G3}:

\begin{thmCpxs}
The following hold for any Grothendieck category $\G$ and a deconstructible class $\F \subseteq \G$:
\begin{enumerate}
\item $\Cpx\F$ is deconstructible in $\Cpx\G$. %Moreover, there is a set $\Q$ of bounded below complexes from $\Cpx\F$ \st $\Cpx\F = \Filt\Q$.

\item $\tilde\F$ is deconstructible in $\Cpx\G$. %Moreover, there is a set $\U$ of bounded above complexes from $\tilde\F$ \st $\tilde\F = \Filt\U$.

\item If $\F$ is a generating class in $\G$, then $\dg\F$ is deconstructible in $\Cpx\G$. %Moreover, $X \in \dg\F$ \iff $X$ is a summand of a complex filtered by the stalk complexes of the form $F[n]$ with $F \in \F$ and $n \in \mathbb{Z}$.
\end{enumerate}
\end{thmCpxs}

\smallskip

Finally, we mention a general structure result which may be of interest by itself. It was observed (though not yet published at the time of writing of this paper) by Enochs for certain deconstructible classes and used to give an alternative proof that these classes were precovering. The result roughly says that we can in some sense limit the length of the transfinite extensions (see \S\ref{sec:bounding length} for more explanation). Here, we denote by $\Sum\clS$ the class of all coproducts of copies of objects in $\clS$.

\begin{thmBoundingLength}
Let $\G$ be a Grothendieck category, $\clS$ a set of objects and $\F$ the closure of $\clS$ under transfinite extensions. Then there exists an infinite regular cardinal $\kappa = \kappa(\G,\clS)$ \st any $X \in \F$ is a transfinite extension of objects from $\Sum\clS$ of length $\le\kappa$.
\end{thmBoundingLength}

After giving (hopefully) motivating relations of deconstructible classes to some remarkable recent work of other authors, we are going to prove in the rest of the paper the results above for which we did not give a reference.

\subsection*{Acknowledgments}

This work was closely related to and stimulated by work on~\cite{SaSt}. I would like to thank Manuel Saor\'\i{}n, my co-author in~\cite{SaSt}, for useful discussions and especially for suggesting Proposition~\ref{prop:deconstr_tildeF}. I would also like to thank Jan Trlifaj for communicating the unpublished result by Enochs in \S\ref{sec:bounding length} and for stimulating discussions.

% ------------------------------------------------------------------------------
\section{Preliminaries}
\label{sec:prelim}

In this paper, we reserve the notation $\G$ for a Grothendieck category, that is, an abelian category with exact direct limits and a generator. Given an infinite regular cardinal $\kappa$, recall that a direct limit in $\G$ is called \emph{$\kappa$-direct} if the indexing set $I$ of the direct system is \emph{$\kappa$-directed}. That is, each subset of $I$ of cardinality $<\kappa$ has an upper bound in $I$.

An object $X \in \G$ is called \emph{$<\kappa$-presentable} if the functor $\Hom_\G(X,-): \G \to \Ab$ preserves $\kappa$-direct limits. An object $X \in \G$ is called \emph{$<\kappa$-generated} if the functor $\Hom_\G(X,-)$ preserves all $\kappa$-direct limits with all morphisms in the direct system being monomorphisms. We refer to~\cite{AR,GaU,S} or~\cite[App. A]{G1} for more information on these notions, and point out that the distinction between direct limits and filtered colimits used in the references is inessential because of~\cite[1.5]{AR}. A Grothendieck category $\G$ is called \emph{locally $<\kappa$-presentable} if it has a generating set $\clS$ consisting of $<\kappa$-presentable objects. In all the cases above, we will use the word ``finitely'' instead of ``$<\aleph_0$''.
Using this terminology, it is well-known that for a unital ring $R$ and $\G = \Mod{R}$, the notions of $<\kappa$-presentable and $<\kappa$-generated objects coincide with usual $<\kappa$-presented and $<\kappa$-generated modules, respectively (see~\cite[\S\S 24.10 and 25.2]{Wis} for $\kappa=\aleph_0$), and $\G$ is locally finitely presentable.

Given an object $X \in \G$ and monomorphisms $i: Y \to X$ and $i': Y' \to X$, we call $i$ and $i'$ equivalent if there is a (unique) isomorphism $f: Y \to Y'$ \st $i = i'f$. Equivalence classes of monomorphisms $Y \to X$ are called \emph{subobjects} of $X$ and, abusing the notation as usual, denoted by $Y \subseteq X$. If $i: Y \to X$ and $j: Z \to X$ are morphisms representing subobjects, we write $Z \subseteq Y$ if there is $f: Z \to Y$ \st $j = if$. It is well known that the set $\Subobj(X)$ of all subobjects of $X$ together with the relation $\subseteq$ forms a complete lattice. The meets and joins are called \emph{intersections} and \emph{sums} of subobjects as usual. If $(Y_i \mid i \in I)$ is a directed system of subobjects, we call their join the \emph{direct union} and denote it by $\bigcup_{i \in I} Y_i$.

For convenience of the reader, we have summarized selected properties of objects of Grothendieck categories and their lattices of subobjects in Appendix~\ref{sec:properties of Grothendieck categories}.

\smallskip

As our motivation stems from construction of suitable model structures for tackling the derived category $\Der\G$, we shall also work with categories of complexes over $\G$. We denote by $\Cpx\G$ the usual abelian category of complexes over $\G$, that is, the objects of $\Cpx\G$ are chain complexes of the form
$$
X: \qquad
\dots \la X^{n-1} \overset{d^{n-1}}\la X^n \overset{d^n}\la X^{n+1} \to \dots,
\qquad
$$
and morphisms are chain complex maps. The complex shifted by $n$ positions to the left with the signs of the differentials correspondingly changed will be denoted by $X[n]$, and the cycle, boundary and homology objects at the $n$-th component will be denoted by $Z^n(X)$, $B^n(X)$ and $H^n(X)$, respectively.
The following easy observation shows that we can stay within the framework of Grothendieck categories:

\begin{lem} \label{lem:complexes over G}
Let $\G$ be a Grothendieck category. Then the category $\Cpx\G$ is also a Grothendieck category.
\end{lem}

\begin{proof}
Clearly, $\Cpx\G$ is abelian with exact filtered colimits, since limits and colimits in $\Cpx{\mathcal G}$ are computed componentwise. Suppose $G \in \G$ is a generator for $\G$. Then the complexes of the shape
$$
\qquad \dots \la 0 \la 0 \la G \stackrel{1_G}\la G \la 0 \la 0 \la \dots,
$$
form a generating set for $\Cpx\G$.
\end{proof}

There is more structure on $\Cpx\G$ which we will need. Namely, instead of all short exact sequences $0 \to X \to Y \to Z \to 0$ of complexes, we sometimes consider only the sequences for which $0 \to X^n \to Y^n \to Z^n \to 0$ splits in $\G$ for each $n \in \Z$. These exact sequences make $\Cpx\G$ an \emph{exact category} in the sense of~\cite[App. A]{Kst}, with the \emph{componentwise split exact structure}, and allow us to define a corresponding variant of the Yoneda $\Ext$ which we denote by $\Ext^n_\CGcs$, to distinguish it from the usual $\Ext$-functor on $\Cpx\G$ which we denote by $\Ext^n_{\Cpx\G}$. We refer to~\cite[App. A]{Kst} and \cite[\S XII.4 and XII.5]{McL} for details.

Let us stress two facts here. First, for each pair $Z,X \in \Cpx\G$, the group $\Ext^1_\CGcs(Z,X)$ is naturally a subgroup of $\Ext^1_{\Cpx\G}(Z,X)$. Second, we have the following well-known lemma:

\begin{lem} \label{lem:stable_hom} \cite[\S1]{Ha}
Assigning to each chain complex morphism $f: Z[-1] \to X$ the componentwise split exact sequence $0 \to X \to C_f \to Z \to 0$, where $C_f$ is the mapping cone of $f$, induces a natural epimorphism
$$ \Hom_{\Cpx\G}(Z[-1],X) \la \Ext^1_\CGcs(Z,X), $$
whose kernel is formed precisely by the null-homotopic morphisms $Z[-1] \to X$.
\end{lem}

%\begin{proof}
% This follows from the fact that $\Cpx\G$ together with the componentwise split exact structure is a so-called Frobenius exact category. That is, there are enough projectives and injectives \wrt the componentwise split short exact sequences and the projectives and injectives coincide. We give~\cite[\S1]{Ha} as a reference.
%\end{proof}

\smallskip

Next, we turn to the central concept---filtrations and filtered objects. Generalizing the corresponding concepts from~\cite[\S3.1]{GT}, we can define them as follows:

\begin{defn} \label{def:filtr}
Let $\clS$ be a class of objects of $\G$. An object $X \in \G$ is called \emph{$\clS$-filtered} if there exists a well-ordered direct system $(X_\alpha, i_{\alpha\beta} \mid \alpha<\beta\le\sigma)$ indexed by an ordinal number $\sigma$ \st
\begin{enumerate}
 \item $X_0 = 0$ and $X_\sigma = X$,
 \item for each limit ordinal $\mu\le\sigma$, the colimit of the subsystem $(X_\alpha, i_{\alpha\beta} \mid \alpha<\beta<\mu)$ is precisely $X_\mu$, the colimit morphisms being $i_{\alpha\mu}: X_\alpha \to X_\mu$,
 \item $i_{\alpha\beta}: X_\alpha \to X_\beta$ is a monomorphism in $\G$ for each $\alpha<\beta\le\sigma$,
 \item $\Coker i_{\alpha,\alpha+1} \in \clS$ for each $\alpha<\sigma$.
\end{enumerate}
The direct system $(X_\alpha, i_{\alpha\beta})$ is then called an \emph{$\clS$-filtration} of $X$.
The class of all $\clS$-filtered objects in $\G$ is denoted by $\Filt\clS$.
\end{defn}

Roughly speaking, $\Filt\clS$ is the class of all transfinite extensions of objects of $\clS$. We will often consider all $X_\alpha$ as subobjects of $X$, in which case condition (2) translates to: $X_\mu = \bigcup_{\alpha<\mu} X_\alpha$ for each limit ordinal $\mu\le\sigma$. The key notion here comes up when $\clS$ is a \emph{set} of objects rather than just a class.

\begin{defn} \label{def:deconstructible}
A class $\F$ of objects in $\G$ is called \emph{deconstructible} if there is a set $\clS$ \st $\F = \Filt\clS$.
\end{defn}

\begin{rem}
In the literature, it is sometimes only required that there be a set $\clS \subseteq \F$ \st each object of $\F$ is $\clS$-filtered, that is, $\F \subseteq \Filt\clS$. We refer for example to~\cite{Tr}. However, we need the equality between $\F$ and $\Filt\clS$ for some results. Luckily, deconstructible classes occurring in practice usually seem to have this property.
\end{rem}

Let us establish some elementary closure properties of deconstructible classes:

\begin{lem} \label{lem:deconstr_filt}
Let $\F$ be a deconstructible class in $\G$. Then any $\F$-filtered object of $\G$ belongs to $\F$, so $\F = \Filt\F$. In particular, $\F$ is closed under taking coproducts and extensions.
\end{lem}

\begin{proof}
Assume that $\F = \Filt\clS$ for some set $\clS$ and that $X \in \G$ is an $\F$-filtered object. That is, there is an $\F$-filtration $(X_\alpha, i_{\alpha\beta} \mid \alpha<\beta\le\sigma)$ with $X_\sigma = X$. We claim that it is possible to refine this filtration to an $\clS$-filtration. To see this, fix $\alpha<\sigma$, denote $F_\alpha = \Coker i_{\alpha,\alpha+1}$, and consider an $\clS$-filtration $(G_\gamma, j_{\gamma\delta} \mid \gamma<\delta\le\tau)$ of $F_\alpha$. Using the pull back diagrams
$$
\begin{CD}
0   @>>>   X_\alpha       @>{m_{0\gamma}}>>         Y_\gamma     @>>>   G_\gamma   @>>>   0\phantom{,}     \\
@.            @|                                      @VVV         @VV{j_{\gamma\tau}}V                    \\
0   @>>>   X_\alpha   @>{i_{\alpha,\alpha+1}}>>   X_{\alpha+1}   @>>>   F_\alpha   @>>>   0,               \\
\end{CD}
$$
it is straightforward to construct a direct system $(Y_\gamma, m_{\gamma\delta} \mid \gamma<\delta\le\tau)$ satisfying (2)--(4) of Definition~\ref{def:filtr}, and \st $Y_0 = X_\alpha$ and $Y_\tau = X_{\alpha+1}$. Now, we can for each $\alpha < \tau$ insert such a direct system between $X_\alpha$ and $X_{\alpha+1}$. In this way, we obtain an $\clS$-filtration for $X$. Hence $X \in \F$.

The fact that $\mathcal{F}$ is closed under extensions is then clear. Finally, given a family $(F'_\gamma \mid \gamma\in I)$ of objects of $\F$, we can assume that $I=\sigma$ is an ordinal. Then there is a filtration $(X_\alpha, i_{\alpha\beta} \mid \alpha<\beta\le\sigma)$ for $\bigoplus_{\gamma \in I} F'_\gamma$, where $X_\alpha=\bigoplus_{\gamma<\alpha} F'_\gamma$ and the morphisms $i_{\alpha\beta}$ are the canonical split monomorphisms.
\end{proof}

To finish the section, we state a connection between filtrations and the $\Ext$-functor. The proposition below is usually used in connection with so-called cotorsion pairs, a concept which we will not need as such in this paper. The adapted version is then as follows:

\begin{prop} \label{prop:cotorsion_and_filt}
Let $\G$ be a Grothendieck category and $\clS$ be a generating set of objects. Let us define the following classes in $\G$:
\begin{itemize}
 \item $\C = \{C \in \G \mid \ExtG(S,C) = 0 \textrm{ for each } S \in \clS\},$
 \item $\F = \{F \in \G\, \mid \ExtG(F,C) = 0 \textrm{ for each } C \in \C\}.$
\end{itemize}
Then $\F$ coincides with the class of all direct summands of objects of $\Filt\clS$.
\end{prop}

\begin{proof}
The statement is a direct consequence of~\cite[Lemma 3.6]{G1} and~\cite[Lemma 4.3]{EGPT}. Alternatively, a more general version for certain exact categories is given in~\cite[Corollary 2.14]{SaSt}, while a more specialized version for module categories has been proved in~\cite[Corollary 3.2.4]{GT}.
\end{proof}

\begin{rem}
A non-trivial consequence proved later on in Proposition~\ref{prop:add_deconstr}(1) is that $\F$ is a deconstructible class.
\end{rem}

% ------------------------------------------------------------------------------
\section{The Hill lemma for Grothendieck categories}
\label{sec:hill}

An important technical tool for dealing with transfinite filtrations in module categories is known as the Generalized Hill Lemma, cf.~\cite[Theorem 4.2.6]{GT} or~\cite[Theorem 6]{StT}. This result, whose idea is due to Hill~\cite{Hill} and versions of which appeared in~\cite{EFS,FL}, roughly says that if we have a single filtration of a module, we automatically get a large family of filtrations.
The key point here is that when overcoming some technical details, a completely analogous result holds for all Grothendieck categories.

\begin{thm}[Generalized Hill Lemma] \label{thm:generalized hill lemma}
Let $\kappa$ be an infinite regular cardinal and $\G$ a locally $<\kappa$-presentable Grothendieck category. Suppose that $\clS$ is a set of $<\kappa$-presentable objects and that $X$ is the union of an $\clS$-filtration
$$
0 = X_0 \subseteq X_1 \subseteq X_2 \subseteq \dots
\subseteq X_\alpha \subseteq X_{\alpha+1} \subseteq \dots
\subseteq X_\sigma = X
$$
for some ordinal $\sigma$. Then there is a complete sublattice $\clL$ of $\big(\Pow(\sigma),\cup,\cap \big)$ and
$$ \ell: \clL \la \Subobj(X) $$
which assigns to each $S \in \clL$ a subobject $\ell(S)$ of $X$, \st the following hold:
\begin{enumerate}
\item[(H1)] For each $\alpha \leq \sigma$ we have $\alpha = \{\gamma \mid \gamma<\alpha \} \in \clL$ and $\ell(\alpha) = X_\alpha$.

\item[(H2)] If $(S_i)_{i \in I}$ is any family of elements of $\clL$, then $\ell(\bigcup S_i) = \sum \ell(S_i)$ and $\ell(\bigcap S_i) = \bigcap \ell(S_i)$. In particular, $\ell$ is a complete lattice homomorphism from $(\clL,\cup,\cap)$ to the lattice $(\Subobj(X), \Sigma, \cap)$ of subobjects of $X$.

\item[(H3)] If $S,T \in \clL$ are \st $S \subseteq T$, then the object $N = \ell(T)/\ell(S)$ is $\clS$-filtered. More precisely, there is an $\clS$-filtration $(N_\beta \mid \beta\le\tau)$ and a bijection  $b: T \setminus S \to \tau$ $(= \{\beta \mid \beta<\tau\})$ \st $X_{\alpha+1}/X_\alpha \cong N_{b(\alpha)+1}/N_{b(\alpha)}$ for each $\alpha \in T \setminus S$.

\item[(H4)] For each $<\kappa$-generated subobject $Y \subseteq X$, there is $S \in \clL$ of cardinality $<\kappa$ (so $\ell(S)$ is $<\kappa$-presentable by \emph{(H3)}) \st $Y \subseteq \ell(S) \subseteq X$.
\end{enumerate}
\end{thm}

\begin{rem} \label{rem:generalized hill lemma}
Although somewhat lengthy, the statement is very natural. To understand the motivation, imagine the following simple model situation. Let $\G = \Mod{R}$, $\clS=\{R\}$ and $X$ be a module with an $\clS$-filtration $(X_\alpha \mid \alpha\le\sigma)$. Clearly, $X \cong R^{(\sigma)}$ is a free module and we can identify $X_\alpha = R^{(\alpha)}$ for each $\alpha \le \sigma$. In this situation, we have the obvious complete lattice homomorphisms $\ell$ from $\big(\Pow(\sigma),\cup,\cap \big)$ (= the power set of $\sigma$, when $\sigma$ is viewed as the set of all smaller ordinals) to the lattice of submodules of $X$ which is given by $\ell(S) = R^{(S)}$. The relation between elements $S \in \Pow(\sigma)$ and $\ell(S) \subseteq X$ is very well understood and, in particular, allows us to construct many other $\clS$-filtrations of $X$. The somewhat surprising fact is that much of this setup is preserved for general filtrations in Grothendieck categories.
\end{rem}

We devote a considerable part of this section to the proof, which is strongly inspired by the arguments used for~\cite[Theorem 4.2.6]{GT} or~\cite[Theorem 6]{StT}. Let us define $\clL$ and $\ell$ first. Given the $\clS$-filtration $(X_\alpha \mid \alpha \le \sigma)$ of $X$, it is not difficult to see that we can fix a family $(A_\alpha \mid \alpha<\sigma)$ of $<\kappa$-generated subobjects of $X$ \st $X_{\alpha+1} = X_\alpha + A_\alpha$ for each $\alpha<\sigma$. Then one calls a subset $S \subseteq \sigma$ \emph{closed} if every $\alpha \in S$ satisfies
$$ X_\alpha \cap A_\alpha \subseteq \sum_{^{\gamma \in S,}_{\gamma < \alpha}} A_\gamma. $$
Next we define
$$ \clL = \{ S \subseteq \sigma \mid S \textrm{ is closed} \} $$
and for any subset $S \subseteq \sigma$ we denote $\ell(S) = \sum_{\alpha \in S} A_\alpha$. The restriction of $\ell$ to $\clL$ will be the map $\ell$ from the statement of Theorem~\ref{thm:generalized hill lemma}. Let us establish a few basic properties of the just defined concepts.

\begin{lem} \label{lem:hill intersections - prelim}
If $S \subseteq \sigma$ is closed, then $\ell(S) \cap X_\alpha = \ell(S \cap \alpha)$ for each $\alpha < \sigma$.
\end{lem}

\begin{proof}
We prove by induction that for each $\beta \le \sigma$ we have
$$ \ell(S \cap \beta) \cap X_\alpha = \ell\big((S \cap \beta) \cap \alpha\big). $$
This is clear for $\beta\le\alpha$, and for limit ordinals $\beta$ it follows from the fact that the lattice of subobjects of $X$ is upper continuous; see Lemma~\ref{lem:upper cont and modular}. Assume finally that $\beta = \delta+1$ and $\delta\ge\alpha$. Since $\ell(S \cap \alpha) \subseteq X_\alpha \subseteq X_\delta$, we are left to show
$$
\big(\ell(S \cap \beta) \cap X_\alpha\big) \cap X_\delta = \ell(S \cap\alpha) \cap X_\delta.
$$
Here we distinguish two cases. Either $\delta\not\in S$ and we simply use the inductive hypothesis, or $\delta\in S$. In the latter case, we have
\begin{multline*}
\big(\ell(S \cap \beta) \cap X_\alpha\big) \cap X_\delta =
\Big(\big(\ell(S \cap \delta) + A_\delta\big) \cap X_\delta\Big) \cap X_\alpha =
\\ =
\Big(\ell(S \cap \delta) + \big(A_\delta \cap X_\delta\big)\Big) \cap X_\alpha =
\ell(S \cap \delta) \cap X_\alpha =
\ell(S \cap \alpha) =
\ell(S \cap \alpha) \cap X_\delta,
\end{multline*}
using that $\Subobj(X)$ is modular, that $S$ is closed, and the inductive hypothesis.
\end{proof}

Next we observe that under some conditions, $\ell$ commutes with intersections.

\begin{lem} \label{lem:intersections}
If $(S_i \mid i \in I)$ is a family of elements of $\clL$, then $\ell\big(\bigcap_{i \in I} S_i\big) = \bigcap_{i \in I} \ell(S_i)$.
\end{lem}

\begin{proof}
Clearly $\ell\big(\bigcap_{i \in I} S_i\big) \subseteq \bigcap_{i \in I} \ell(S_i)$. Suppose for the moment that the other inclusion does not hold and let $\beta\le\sigma$ be minimal \st
$$
\ell\Big(\bigcap_{i \in I} S_i\Big) \cap X_\beta \subsetneqq
\bigcap_{i \in I} \ell(S_i) \cap X_\beta
$$
Obviously $\beta>0$ and $\beta$ cannot be a limit ordinal since $\Subobj(X)$ is upper continuous. Thus, $\beta = \delta+1$ for some $\delta$ and $\ell\left(\bigcap_{i \in I} S_i\right) \cap X_\delta = \bigcap_{i \in I} \ell(S_i) \cap X_\delta$. It follows that $\bigcap_{i \in I} \ell(S_i) \cap X_\delta \subsetneqq \bigcap_{i \in I} \ell(S_i) \cap X_\beta$, so using Lemma~\ref{lem:hill intersections - prelim} we get
$$
\ell(S_i \cap \delta) =
\ell(S_i) \cap X_\delta \subsetneqq \ell(S_i) \cap X_\beta =
\ell(S_i \cap \beta)
$$
for each $i \in I$. In particular, $\delta$ must belong to $S_i$ for each $i \in I$. Hence $A_\delta \subseteq \ell\left(\bigcap_{i \in I} S_i\right) \subseteq \bigcap_{i \in I} \ell(S_i)$ and we have
$$
\bigg(\ell\Big(\bigcap_{i \in I} S_i\Big) \cap X_\beta\bigg) + X_\delta =
X_\beta =
\Big(\bigcap_{i \in I} \ell(S_i) \cap X_\beta\Big) + X_\delta.
$$
Since $\Subobj(X)$ is modular, the last equality together with $\ell\left(\bigcap_{i \in I} S_i\right) \cap X_\delta = \bigcap_{i \in I} \ell(S_i) \cap X_\delta$ implies that $\ell\left(\bigcap_{i \in I} S_i\right) \cap X_\beta = \bigcap_{i \in I} \ell(S_i) \cap X_\beta$, which is a contradiction.
\end{proof}

Now we are able to prove (H2) of Theorem~\ref{thm:generalized hill lemma}.

\begin{lem} \label{lem:complete lattice morphism}
Let $(S_i \mid i \in I)$ be a family of elements of $\clL$. Then both $\bigcup_{i \in I} S_i$ and $\bigcap_{i \in I} S_i$ belong to $\clL$. In particular, $\ell\colon \clL \to \Subobj(X)$ is a complete lattice homomorphism.
\end{lem}

\begin{proof}
It is easy to see that $S = \bigcup_{i \in I} S_i \in \clL$, since for any $\alpha \in S$ there is $i_0 \in I$ \st $\alpha \in S_{i_0}$ and we have
$$
X_\alpha \cap A_\alpha \subseteq
\sum_{^{\gamma \in S_{i_0},}_{\;\gamma < \alpha}} A_\gamma \subseteq
\sum_{^{\gamma \in S,}_{\gamma < \alpha}} A_\gamma.
$$
If on the other hand $T = \bigcap_{i \in I} S_i$ and $\alpha \in T$, then the assumptions and Lemma~\ref{lem:intersections} yield
$$
X_\alpha \cap A_\alpha \subseteq
\bigcap_{i \in I} \ell(S_i \cap \alpha) =
\ell(T \cap \alpha) =
\sum_{^{\gamma \in T,}_{\gamma < \alpha}} A_\gamma,
$$
so $T \in \clL$, too. Finally, $\ell: \clL \to \Subobj(X)$ is easily seen to be a complete lattice homomorphism using Lemma~\ref{lem:intersections}.
\end{proof}

At this point we are in a position to complete the proof of Theorem~\ref{thm:generalized hill lemma}. Before doing so, we point up that the only assumption we used to prove Lemmas~\ref{lem:hill intersections - prelim} to~\ref{lem:complete lattice morphism} was that $\Subobj(X)$ was a complete modular upper continuous lattice.

\begin{proof}[Proof of Theorem~\ref{thm:generalized hill lemma}]
Lemma~\ref{lem:complete lattice morphism} tells us that $\clL$ is a complete sublattice of the lattice $\big(\Pow(\sigma),\cup,\cap\big)$ and that $\ell\colon \clL \to \Subobj(X)$ is a complete lattice homomorphism. This proves (H2). Clearly, $\alpha \in \clL$ and $\ell(\alpha) = X_\alpha$ for each $\alpha\le\sigma$, so (H1) follows.

To establish (H3), we proceed exactly as in the proof of~\cite[Theorem 6(H3)]{StT}. Namely, we consider the filtration $(\bar F_\alpha \mid \alpha\le\lambda)$ of $\ell(T)/\ell(S)$, where $\bar F_\alpha$ is defined for each $\alpha\le\sigma$ by
$$ F_\alpha = \ell\big( S \cup (T \cap \alpha) \big) \quad \textrm{ and } \quad \bar F_\alpha = F_\alpha/\ell(S). $$
It follows that for given $\alpha<\sigma$, either $\alpha \in T \setminus S$ and
$$
\bar F_{\alpha+1}/\bar F_\alpha \cong
(F_\alpha + A_\alpha) / F_\alpha \cong
A_\alpha / (A_\alpha \cap F_\alpha) =
A_\alpha / (A_\alpha \cap X_\alpha) \cong
X_{\alpha+1}/X_\alpha,
$$
or $\bar F_{\alpha+1} = \bar F_\alpha$. The filtration $(N_\beta \mid \beta\le\tau)$ is obtained from $(\bar F_\alpha \mid \alpha\le\sigma)$ just by removing repetitions, and $b: T \setminus S \to \tau$ is defined in the obvious way.

Finally, (H4) is proved similarly as in~\cite[Theorem 6]{StT} on pages 310/311. Given a $<\kappa$-generated subobject $Y \subseteq X$, it is not difficult to see that there is a (not necessarily closed) subset $S' \subseteq \sigma$ of cardinality $<\kappa$ \st $Y \subseteq \ell(S')$. We will prove that each $S' \subseteq \sigma$ of cardinality $<\kappa$ is contained in a closed subset $S \subseteq \sigma$ of cardinality $<\kappa$. In view of Lemma~\ref{lem:complete lattice morphism}, it suffices to prove this for singletons $S = \{\beta\}$. This is achieved by induction on $\beta < \sigma$. If $\beta < \kappa$, we simply take $S = \beta+1$. Otherwise, Lemma~\ref{lem:size of kernels} applied on the short exact sequence
$$ 0 \la X_\beta \cap A_\beta \la A_\beta \la X_{\beta+1}/X_\beta \la 0 $$
tells us that $X_\beta \cap A_\beta$ is $<\kappa$-generated. By induction, there is $S_0 \in \clL$ of cardinality $<\kappa$ \st $X_\beta \cap A_\beta \subseteq \ell(S_0)$. We claim that $S = S_0 \cap \{\beta\}$ is a set we want. It is enough to show that $S$ is closed and it suffices to check the definition only for $\beta$. However, we have $X_\beta \cap A_\beta \subseteq X_\beta \cap \ell(S_0) = \ell(S_0 \cap \beta) \subseteq \sum_{\gamma \in S, \gamma<\beta} A_\gamma$ and the claim is proved. To finish, note that $\ell(S)$ is $<\kappa$-presentable by (H3) and Corollary~\ref{cor:small objects closed under small filtrations}.
\end{proof}

Before giving applications, let us point out that the image of $\ell$ is a complete \emph{distributive} sublattice of $\Subobj(X)$, while $\Subobj(X)$ itself is usually only modular. Now we start illustrating the potential of the theorem by proving certain non-trivial consequences.

\smallskip

We start with a relation to so-called Kaplansky classes (see~\cite{EnoLR,G1}). Kaplansky classes have been explicitly or implicitly used for proving approximation properties of flat modules or sheaves in various settings \cite{AEGO,BEE,EE,EEGO,EnoOy,EnoLR}, a fact which Gillespie~\cite{G1} and Hovey~\cite{H3} later applied in a crucial way to constructing monoidal model structures for complexes of sheaves.

\begin{defn} \label{def:kaplansky}
Let $\F \subseteq \G$ be a class of objects and $\kappa$ a regular cardinal. Then $\F$ is said to be a \emph{$\kappa$-Kaplansky class} if for any $F \in \F$ and a $<\kappa$-generated subobject $X \subseteq F$, there exists a $<\kappa$-presentable subobject $Y$ of $F$ \st $X \subseteq Y \subseteq F$ and $Y,F/Y \in \F$.
We say that $\F$ is a \emph{Kaplansky class} if it is $\kappa$-Kaplansky for some regular cardinal $\kappa$.
\end{defn}

We prove as an easy corollary of Theorem~\ref{thm:generalized hill lemma} that a deconstructible class is always a Kaplansky class, a result which for instance considerably simplifies Gillespie's arguments (especially those in~\cite[\S4]{G1}). For module categories an analogous result has been stated as~\cite[Lemmas 6.7 and 6.9]{HT}, while for categories of quasi-coherent sheaves this was implicitly proved in~\cite{EGPT}.

\begin{cor} \label{cor:dec_kapl}
Let $\F$ be a class of objects in a Grothendieck category $\G$. Then the following hold:
\begin{enumerate}
\item If $\F$ is a deconstructible class, then there is an infinite regular cardinal $\kappa$ \st $\F$ is a $\lambda$-Kaplansky class for each $\lambda\ge\kappa$.
\item If $\F$ is a Kaplansky class and closed under taking direct limits, then it is deconstructible.
\end{enumerate}
\end{cor}

\begin{proof}
(1) is an easy consequence of Theorem~\ref{thm:generalized hill lemma}(H4) and (H3) and Corollary~\ref{cor:small objects closed under small filtrations}, while part (2) is rather standard and we refer to the proof of~\cite[Lemma 4.3]{G1} for a detailed argument.
\end{proof}

\begin{rem}
The condition of $\F$ being closed under direct limits in Corollary~\ref{cor:dec_kapl}(2) cannot be easily removed. A recent result of Herbera and Trlifaj~\cite[Example 6.8]{HT} shows that the class of flat Mittag-Leffler modules over the endomorphism ring of an infinite dimensional vector space is Kaplansky, but not deconstructible.
\end{rem}

\smallskip

Further, we show that deconstructibility is kept under some natural operations on classes---under the closure under direct summands and under set-indexed intersections.

\begin{prop} \label{prop:add_deconstr}
Given an uncountable regular cardinal $\kappa$ and a locally $<\kappa$-presentable Grothendieck category $\G$, the following hold:
\begin{enumerate}
\item Let $\E = \Filt\clS$ in $\G$, where $\clS$ is a set of $<\kappa$-presentable objects, and let $\F$ be the class consisting of all direct summands of objects from $\E$. Then there is a set $\clS'$ of $<\kappa$-presentable objects \st $\F = \Filt{\clS'}$.

\item Let $(\E_i \mid i \in I)$ be a collection of classes of objects of $\G$ \st $\card{I}<\kappa$. Suppose that for each $i \in I$, there is a set $\clS_i$ of $<\kappa$-presentable objects \st $\E_i = \Filt{\clS_i}$. Then there is a set $\clS'$ of $<\kappa$-presentable objects \st $\bigcap_{i \in I} \E_i = \Filt{\clS'}$.
\end{enumerate}
\end{prop}

\begin{proof}
(1) The argument is analogous to~\cite[Lemma 9]{StT}. Suppose that $X \in \F$, that is, there is $Y \in \G$ \st $Z = X\oplus Y$ is $\clS$-filtered. We denote by $\pi_X: Z \to X$ and $\pi_Y: Z \to Y$ the corresponding split projections and by $\clS'$ a representative set of all $<\kappa$-presentable objects of $\F$, and we show that $X$ is $\clS'$-filtered.

Let $\ell: \clL \to \Subobj(Z)$ be a complete lattice homomorphism as in Theorem~\ref{thm:generalized hill lemma} and let $\clH = \{\ell(S) \mid S \in \clL\}$. We first show that there is a filtration $(Z_\alpha \mid \alpha\le\sigma)$ of $Z$ \st
\begin{enumerate}
\item[(i)] $Z_\alpha \in \clH$,
\item[(ii)] $Z_\alpha = \pi_X(Z_\alpha) + \pi_Y(Z_\alpha)$, and
\item[(iii)] $Z_{\alpha+1}/Z_\alpha$ is $<\kappa$-presentable for each $\alpha<\sigma$.
\end{enumerate}
To this end, we put $Z_0 = 0$ and $Z_\alpha = \bigcup_{\gamma<\alpha} Z_\gamma$ for limit ordinals. Suppose we have constructed $Z_\alpha \subsetneqq Z$ and wish to construct $Z_{\alpha+1}$. Note that $\G$ being locally $<\kappa$-presentable ensures that there is a $<\kappa$-generated subobject $W \subseteq Z$ \st $W \not\subseteq Z_\alpha$. Using Theorem~\ref{thm:generalized hill lemma}(H4), we find $S \in \clL$ of cardinality $<\kappa$ \st $W \subseteq \ell(S)$. Denoting $W' = \ell(S)$, $Q_0 = Z_\alpha + W'$ and combining (H2) and (H3) with Corollary~\ref{cor:small objects closed under small filtrations}, we observe that $Q_0 \in \clH$, $Z_\alpha + W \subseteq Q_0$ and $W', Q_0/Z_\alpha$ are $<\kappa$-presentable. We also have:
$$
Q_0 \subseteq
\pi_X(Q_0) + \pi_Y(Q_0) =
\pi_X(Z_\alpha+W') + \pi_Y(Z_\alpha+W') =
Z_\alpha + \big(\pi_X(W') + \pi_Y(W')\big).
$$
Since $\pi_X(W') + \pi_Y(W')$ is $<\kappa$-generated, we can use the same argument as above to find $Q_1 \in \clH$ \st $\pi_X(Q_0) + \pi_Y(Q_0) \subseteq Q_1$ and $Q_1/Z_\alpha$ is $<\kappa$-presentable. By proceeding inductively, we find a chain of subobjects $Q_0 \subseteq Q_1 \subseteq Q_2 \subseteq \dots$ from $\clH$ \st $Q_i \subseteq \pi_X(Q_i) + \pi_Y(Q_i) \subseteq Q_{i+1}$ and $Q_i/Z_\alpha$ is $<\kappa$-presentable for each $i<\omega$. It is easy to see that $Z_{\alpha+1} = \bigcup_{i<\omega} Q_i$ satisfies conditions (i)--(iii) above.

Having constructed the filtration $(Z_\alpha \mid \alpha\le\sigma)$, it is not difficult to prove that
$$
Z_{\alpha+1}/Z_\alpha \cong
\pi_X(Z_{\alpha+1})/\pi_X(Z_\alpha) \oplus \pi_Y(Z_{\alpha+1})/\pi_Y(Z_\alpha)
$$
for each $\alpha<\sigma$; we refer to~\cite[p. 314]{StT}. It follows that $(\pi_X(Z_\alpha) \mid \alpha\le\sigma)$ is an $\clS'$-filtration of $X$, as desired.

(2) First we claim that there is a limit ordinal $\tau<\kappa$ and a map $b: \tau \to I$ \st $b^{-1}(i)$ is an unbounded subset in $\tau$ for each $i \in I$. Without loss of generality, we may assume that $I = \mu$ is an infinite cardinal number. Let $\tau$ be the ordinal type of $\omega \times \mu$ with the lexicographical ordering, and denote by $c: \tau \to \omega \times \mu$ the order isomorphism. Then we can define $b: \tau \to \mu = I$ as the composition of $c$ with the projection $\omega \times \mu \to \mu$ onto the second component. It is straightforward to check that $b$ has the required properties and the claim is proved.

Suppose now that $X \in \bigcap_{i \in I} \E_i$. Then we can for each $i \in I$ fix an $\clS_i$-filtration of $X$ and, using Theorem~\ref{thm:generalized hill lemma}, a corresponding complete lattice homomorphism $\ell_i: \clL_i \to \Subobj(X)$. Denote $\clH_i = \{ \ell_i(S) \mid S \in \clL_i \}$ for each $i \in I$, and let $\clH = \bigcap_{i \in I} \clH_i$. It is easy to see from Theorem~\ref{thm:generalized hill lemma}(H2) and~(H3) that $\clH$ is closed under taking arbitrary sums and intersections, and that for each $N,P \in \clH$ with $N \subseteq P$ we have $P/N \in \bigcap_{i \in I} \E_i$.
Next we will inductively construct a filtration $(X_\alpha \mid \alpha<\sigma)$ of $X$ \st
\begin{enumerate}
\item[(a)] $X_\alpha \in \clH$, and
\item[(b)] $X_{\alpha+1} / X_\alpha$ is $<\kappa$-presentable for each $\alpha<\sigma$.
\end{enumerate}
If we succeed to do so, it will easily follow that $\bigcap_{i \in I} \E_i = \Filt{\clS'}$ for a representative set $\clS'$ of all $<\kappa$-presentable objects in $\bigcap_{i \in I} \E_i$ (cf.\ also Lemma~\ref{lem:deconstr_filt}). 

Again, we put $X_0 = 0$ and $X_\alpha = \bigcup_{\gamma<\alpha} X_\gamma$ for limit ordinals $\alpha$. Suppose now we have constructed $X_\alpha \subsetneqq X$. Then there is a $<\kappa$-generated subobject $W \subseteq X$ \st $W \not\subseteq X_\alpha$. Using Theorem~\ref{thm:generalized hill lemma}(H2)--(H4) and Corollary~\ref{cor:small objects closed under small filtrations}, we obtain $Q_0 \in \clH_{b(0)}$ \st $X_\alpha+W \subseteq Q_0$ and $Q_0/X_\alpha$ is $<\kappa$-presentable. Therefore, there is $W' \subseteq X$ which is $<\kappa$-generated and \st $X_\alpha+W' = Q_0$. Using the same argument, we get $Q_1 \in \clH_{b(1)}$ \st $Q_0 = X_\alpha+W' \subseteq Q_1$ and $Q_1/X_\alpha$ is $<\kappa$-presentable. Proceeding further in this way, we construct a chain $(Q_\iota \mid \iota<\tau)$ of subobjects of $X$ \st $Q_\iota \in \clH_{b(\iota)}$ and $Q_\iota/X_\alpha$ is $<\kappa$-presentable for each $\iota < \tau$. It easily follows that $X_{\alpha+1} = \bigcup_{\iota<\tau} Q_\iota$ satisfies conditions (a) and (b) above. This finishes the construction and also the proof of the theorem.
\end{proof}

% ------------------------------------------------------------------------------
\section{Bounding the length of a filtration}
\label{sec:bounding length}

Given a set $\clS$ in a Grothendieck category $\G$, let us denote by $\Sum\clS$ the class of all coproducts of copies of objects from $\clS$. If we now have an $\clS$-filtration $(X_\alpha \mid \alpha\le\sigma)$ of some $X \in \G$, there is of course no obvious bound on $\sigma$. However, we show that we can always rearrange the filtration to have a bound on its length independent of $X$, if we allow the factors to be in $\Sum\clS$ instead of $\clS$. This extends a result (unpublished at the time of writing of this paper) by Enochs.

\begin{thm} \label{thm:bounding length}
Let $\kappa$ be an infinite regular cardinal and $\G$ a locally $<\kappa$-presentable Grothendieck category. Given a class $\clS$ of $<\kappa$-presentable objects and an object $X\in\Filt\clS$, then $X$ has a $(\Sum\clS)$-filtration of the form $(X'_\beta \mid \beta\le\kappa)$. 
\end{thm}

\begin{rem} \label{rem:bounding length}
It is worthwhile to look closer at the case $\kappa=\aleph_0$. The theorem says that given a locally finitely presentable Grothendieck category $\G$, a set $\clS\subseteq\G$ of finitely presentable objects and an $\clS$-filtered object $X$, then $X$ is the union of a countable chain
$$ 0 = X'_0 \subseteq X'_1 \subseteq X'_2 \subseteq X'_3 \subseteq \cdots $$
\st $X'_{n+1}/X'_n \in \Sum\clS$ for each $n\ge 0$. As a simple model situation, we can think of a right noetherian ring, where any semiartinian right module has a countable filtration by semisimple modules.
\end{rem}

\begin{proof}[Proof of Theorem~\ref{thm:bounding length}]
Let $(X_\alpha \mid \alpha\le\sigma)$ be an $\clS$-filtration of $X$ and $\ell: \clL \to \Subobj(X)$ be a complete lattice homomorphism as in Theorem~\ref{thm:generalized hill lemma}. Let us further for each $\alpha<\sigma$ fix a $<\kappa$-generated subobject $A_\alpha \subseteq X$ \st $X_{\alpha+1} = X_\alpha+A_\alpha$. Theorem~\ref{thm:generalized hill lemma}(H4) allows us to fix for each $\alpha<\sigma$ a set $S_\alpha \in \clL$ \st $\alpha \in S_\alpha$, $\card{S_\alpha} < \kappa$ and $A_\alpha \subseteq \ell(S_\alpha)$. By (H1) and (H2), we may also assume that $S_\alpha \subseteq (\alpha+1)$, by possibly passing from $S_\alpha$ to $S_\alpha\cap(\alpha+1)$.

With the notation above, we inductively define a map $\lev\colon \sigma \to \kappa$ by putting
$$
\lev(\alpha) =
\sup\{\lev(\gamma)+1 \mid (\gamma<\alpha) \;\&\; (\gamma \in S_\alpha)\},
$$
and call $\lev(\alpha)$ the \emph{level} of $\alpha$. Here, $\sup\emptyset=0$ by definition. Note that $\lev(\alpha)$ is well-defined since $\kappa$ is regular and $\card{S_\alpha} < \kappa$ for each $\alpha<\sigma$. Let us denote $T_\beta = \{\gamma \mid (\gamma<\sigma) \; \& \; (\lev(\gamma) < \beta) \}$; one readily checks that $T_\beta = \bigcup \{S_\gamma \mid (\gamma<\sigma) \; \& \; (\lev(\gamma) < \beta) \}$. We claim that
$$
X'_\beta = \ell(T_\beta) \quad
\bigg(= \sum_{^{\;\;\;\gamma<\sigma,}_{\lev(\gamma)<\beta}} A_\gamma\bigg)
$$
yields a $(\Sum\clS)$-filtration $(X'_\beta \mid \beta\le\kappa)$ of $X$, as desired. The only non-trivial part is proving that for any fixed $\beta<\kappa$ we have $X'_{\beta+1}/X'_\beta \in \Sum\clS$. We will show more, namely that
$$
X'_{\beta+1}/X'_\beta \cong
\bigoplus_{^{\;\;\;\gamma<\sigma,}_{\lev(\gamma)=\beta}} X_{\gamma+1}/X_\gamma.
$$
Since for each $\delta<\sigma$ we have $X_{\delta+1}/X_\delta \cong A_\delta / (A_\delta \cap X_\delta)$, it is easy to see that it suffices to show for each $\delta<\sigma$ with $\lev(\delta) = \beta$ that
$$
A_\delta \cap X'_\beta = A_\delta \cap X_\delta
\quad \textrm{and} \quad
(A_\delta + X'_\beta) \cap \sum_{^{\;\;\;\gamma<\delta,}_{\lev(\gamma)=\beta}} A_\gamma
\subseteq X'_\beta.
$$ 
For the first part, note that we have $A_\delta \cap X_\delta \subseteq \ell(S_\delta) \cap X_\delta = \ell(S_\delta \cap \delta) \subseteq X'_\beta$ and on the other hand, using Theorem~\ref{thm:generalized hill lemma}(H1) and~(H2) and the fact that $\delta \not\in T_\beta$, also:
$$
A_\delta \cap X'_\beta \subseteq \ell(\delta+1) \cap \ell(T_\beta) =
\ell\big((\delta+1) \cap T_\beta\big) \subseteq \ell(\delta) = X_\delta.
$$
Hence $A_\delta \cap X'_\beta = A_\delta \cap X_\delta$. The second part is similar. Namely, given any $\gamma < \sigma$ of level $\beta$, the only element in $S_\gamma$ whose level is at least $\beta$ is $\gamma$ itself. It follows that $S_\delta \cap S_\gamma \subseteq T_\beta$ for any such $\gamma<\delta$, and
\begin{multline*}
(A_\delta + X'_\beta) \cap \sum_{^{\;\;\;\gamma<\delta,}_{\lev(\gamma)=\beta}} A_\gamma \subseteq
\ell(S_\delta \cup T_\beta) \cap \sum_{^{\;\;\;\gamma<\delta,}_{\lev(\gamma)=\beta}} \ell(S_\gamma) = \\
= \ell\Big((S_\delta \cup T_\beta) \cap \bigcup_{^{\;\;\;\gamma<\delta,}_{\lev(\gamma)=\beta}} S_\gamma\Big) \subseteq
\ell(T_\beta) = X'_\beta.
\end{multline*}
This finishes the proof of the claim and also of the theorem.
\end{proof}

% ------------------------------------------------------------------------------
\section{Deconstructibility for complexes}
\label{sec:cpxs}

In this section, we discuss a crucial point of papers~\cite{G3,G2} by Gillespie on construction of certain nice model structures and also papers~\cite{Nee3,Nee} by Neeman which are focused on existence of certain triangulated adjoint functors. Namely, given a Grothendieck category $\G$ and a deconstructible class $\F \subseteq \G$, one requires that certain classes in $\Cpx\G$ determined by $\F$ also be deconstructible.
% and hence, as it turns out, precovering.
There are various candidates for classes determined by $\F$. In the simplest case, we may take $\Cpx\F$, the full subcategory of $\Cpx\G$ formed by the complexes with components in $\F$. In the above mentioned papers, other classes played an important role, too:

\begin{nota} \label{not:F tilde and dg-F}
Let $\F$ be a class of objects in a Grothendieck category $\G$, and let $\C = \{C \in \G \mid \ExtG(F,C) = 0 \textrm{ for each } F \in \F\}$. Then we denote:
\begin{enumerate}
\item by $\tilde\F$ the class of all acyclic complexes $X \in \Cpx\G$ \st $Z^n(X) \in \F$ for each integer $n$;

\item by $\dg\F$ the class of all complexes $X \in \Cpx\F$ \st every chain complex morphism $X \to Y$ with $Y \in \tilde\C$ is null-homotopic. Here, $\tilde\C$ follows the notation from (1) for $\C$ in place of $\F$.
\end{enumerate}
\end{nota}

Although the definition of $\dg\F$ may seem mysterious at the first sight, it only generalizes Spaltenstein's concept of K-projective complexes~\cite{Spa}, which is reconstructed by taking $\G = \Mod{R}$ and for $\F$ the class of all projective modules.
A slight modification of~\cite[Lemma 3.9]{G3} shows that for an extension closed class $\mathcal{F}$ we always have the inclusions
$$ \tilde\F \subseteq \dg\F \subseteq \Cpx\F. $$

Our main result in this direction then says:

\begin{thm} \label{thm:deconstr_cpxs}
Let $\G$ be a Grothendieck category and $\F$ be a deconstructible class of objects of $\G$. Then the following assertions hold:
\begin{enumerate}
\item $\Cpx\F$ is deconstructible in $\Cpx\G$. More precisely, there is a set $\Q$ of bounded below complexes from $\Cpx\F$ \st $\Cpx\F = \Filt\Q$.

\item $\tilde\F$ is deconstructible in $\Cpx\G$. More precisely, there is a set $\U$ of bounded above complexes from $\tilde\F$ \st $\tilde\F = \Filt\U$.

\item If $\F$ is a generating class in $\G$, then $\dg\F$ is deconstructible in $\Cpx\G$. In this case, each $X \in \dg\F$ is a summand of a complex filtered by the stalk complexes of the form $F[n]$ with $F \in \F$ and $n \in \mathbb{Z}$.
\end{enumerate}
\end{thm}

We will prove the theorem in a few steps, giving a corresponding argument for each of the parts (1), (2) and (3) separately. We start with Theorem~\ref{thm:deconstr_cpxs}(1).

\begin{prop} \label{prop:deconstr_C(F)}
Let $\kappa$ be an infinite regular cardinal and $\G$ a locally $<\kappa$-presentable Grothendieck category. Suppose that $\F \subseteq \G$ is a deconstructible class \st $\F = \Filt\clS$, where $\clS$ is a representative set of all $<\kappa$-presentable objects of $\G$ contained in $\F$. Then each $X \in \Cpx\F$ is filtered by bounded below complexes with components in $\clS$. In particular, $\Cpx\F$ is deconstructible.
\end{prop}

\begin{proof}
Let $\kappa$, $\G$, $\F$ and $\clS$ be as above, and denote by $\Q$ the class of all bounded below complexes over $\G$ with all components in $\clS$. Given $X \in \Cpx\F$, we fix for each component $X^n$ an $\clS$-filtration and for this filtration a complete lattice morphism $\ell_n: \clL_n \to \Subobj(X^n)$ provided by Theorem~\ref{thm:generalized hill lemma}. For each integer $n$, we further put $\clH_n = \{ \ell_n(S) \mid S \in \clL_n \}$.

Using this data, we will construct by induction a $\Q$-filtration $(X_\alpha \mid \alpha\le\sigma)$ of the complex $X$ \st $X^n_\alpha \in \clH_n$ for each $n \in \mathbb{Z}$ and $\alpha\le\sigma$.
We are required to put $X_0 = 0$ and take direct unions of subcomplexes at limit steps. For a non-limit step, assume we have constructed $X_\alpha \subsetneqq X$ and we take an integer $n$ and a $<\kappa$-generated subobject $W \subseteq X^n$ \st $W \not\subseteq X^n_\alpha$. Then we put $X^m_{\alpha+1} = X^m_\alpha$ for $m<n$ and take $X^n_{\alpha+1} \in \clH_n$ \st $X^n_\alpha + W \subseteq X^n_{\alpha+1}$ and $X^n_{\alpha+1}/X^n_\alpha$ is $<\kappa$-presentable and $\clS$-filtered. Note that we can always do this using Theorem~\ref{thm:generalized hill lemma}(H4) and~(H3) and Corollary~\ref{cor:small objects closed under small filtrations}. Further note that, up to isomorphism, $X^n_{\alpha+1}/X^n_\alpha \in \clS$ by Lemma~\ref{lem:deconstr_filt}.

For $m > n$ we proceed by induction. Suppose we have already constructed $X^{m-1}_{\alpha+1}$ \st $X^{m-1}_{\alpha+1}/X^{m-1}_\alpha \in \clS$, up to isomorphism. Then there is a $<\kappa$-generated subobject $W' \subseteq X^{m-1}_{\alpha+1}$ \st $X^{m-1}_{\alpha+1} = X^{m-1}_\alpha+W'$. Since $d^{m-1}_X(W')$ is a $<\kappa$-generated subobject of $X^m$, we can again use Theorem~\ref{thm:generalized hill lemma} to find $X^m_{\alpha+1} \in \clH_m$ \st $X^m_\alpha + d^{m-1}_X(W') \subseteq X^m_{\alpha+1}$ and $X^m_{\alpha+1} / X^m_\alpha$ is isomorphic to an object from $\clS$. This finishes the induction.
It is then easy to check that $X_{\alpha+1}$ is a subcomplex of $X$ and $X_{\alpha+1}/X_\alpha$ is isomorphic to an element of $\Q$.
\end{proof}

The case of $\tilde\F$ (see Notation~\ref{not:F tilde and dg-F}) is similar, but technically more involved.

\begin{prop} \label{prop:deconstr_tildeF}
Let $\kappa$ be an infinite regular cardinal and $\G$ a locally $<\kappa$-presentable Grothendieck category. Suppose that $\F \subseteq \G$ is a deconstructible class \st $\F = \Filt\clS$ for a set $\clS$ of $<\kappa$-presentable objects. Then each $X \in \tilde\F$ is filtered by bounded above complexes from $\tilde\F$ with $<\kappa$-presentable components. In particular, $\tilde\F$ is deconstructible.
\end{prop}

\begin{proof}
Let $\U$ be a representative set for all bounded above complexes in $\tilde\F$ with $<\kappa$-presentable components. We must prove that $\tilde\F = \Filt\U$. In fact, it is easy to see using Lemma~\ref{lem:deconstr_filt} that $\Filt\U \subseteq \tilde\F$, so we are left with proving that any $X \in \tilde\F$ is $\U$-filtered.

Given such $X$, we fix for each cycle object $Z^n(X)$ an $\clS$-filtration and for this filtration a complete lattice morphisms $\ell_n: \clL_n \to \Subobj\big(Z^n(X)\big)$ provided by Theorem~\ref{thm:generalized hill lemma}. For each integer $n$, we further put $\clH_n = \{ \ell_n(S) \mid S \in \clL_n \}$. Now we will inductively construct a filtration $(X_\alpha \mid \alpha\le\sigma)$ of $X$ \st for each $\alpha<\sigma$:
\begin{enumerate}
 \item $X_\alpha$ is an acyclic complex,
 \item $X_{\alpha+1}/X_\alpha$ is bounded above,
 \item $Z^n(X_\alpha) \in \clH_n$ for each $n \in \mathbb{Z}$, and
 \item $Z^n(X_{\alpha+1})/Z^n(X_\alpha)$ is $<\kappa$-presentable for each $n \in \mathbb{Z}$.
\end{enumerate}

Assume for the moment we have such a filtration. Then $X_{\alpha+1}/X_\alpha$ is an acyclic complex for each $\alpha$ by the $3\times 3$ lemma. Since each $Z^n(X_{\alpha+1})/Z^n(X_\alpha)$ is $<\kappa$-presentable, Lemma~\ref{lem:small objects closed under ext} tells us that each component of $X_{\alpha+1}/X_\alpha$ is $<\kappa$-presentable. Using (2), (3) and Theorem~\ref{thm:generalized hill lemma}(H3), it immediately follows that $X_{\alpha+1}/X_\alpha \in \U$, up to isomorphism. Hence $(X_\alpha \mid \alpha\le\sigma)$ is a $\U$-filtration and we are done.

To construct the filtration, we put $X_0 = 0$ and $X_\alpha = \bigcup_{\gamma<\alpha} X_\gamma$ for limit ordinals $\alpha\le\sigma$. Note that $X_\alpha$ is acyclic in the latter case because direct limits are exact in $\G$. For non-limit steps, assume we have constructed $X_\alpha \subsetneqq X$ for some $\alpha$ and we wish to construct $X_{\alpha+1}$. It is easy to see that there is $n \in \mathbb{Z}$ \st $Z^n(X_\alpha) \subsetneqq Z^n(X)$. We put $X_{\alpha+1}^m = X_\alpha^m$ for all $m>n$ and construct $X_{\alpha+1}^m$ inductively for $m \le n$. To this end, suppose that $X_{\alpha+1}^{m+1}$ has been already constructed and denote $Z_{\alpha+1}^{m+1} = X_{\alpha+1}^{m+1} \cap Z^{m+1}(X)$. Since $Z_{\alpha+1}^{m+1} / Z^{m+1}(X_\alpha)$ is $<\kappa$-presentable (see requirement (4) above), there is a $<\kappa$-generated subobject $W \subseteq X^m$ \st $Z_{\alpha+1}^{m+1} = Z^{m+1}(X_\alpha) + d_X^m(W)$. If $m=n$, we in addition take $W$ so that $W \not\subseteq Z^n(X_\alpha)$, which will ensure further in the construction that $X^n_\alpha \subsetneqq X^n_{\alpha+1}$. Denoting $K = (X_\alpha^{m} + W) \cap Z^m(X)$, we get the following diagram with exact rows:
$$
\begin{CD}
0 @>>> Z^m(X_\alpha) @>{\subseteq}>> X^m_\alpha   @>{{d_X^m\restriction X^m_\alpha}}>>     Z^{m+1}(X_\alpha)   @>>> 0    \\
@.   @V{\subseteq}VV               @VV{\subseteq}V                                          @VV{\subseteq}V              \\
0 @>>>       K       @>{\subseteq}>> X^m_\alpha+W @>{{d_X^m\restriction (X^m_\alpha+W)}}>> Z_{\alpha+1}^{m+1}  @>>> 0
\end{CD}
$$
Using the $3 \times 3$ lemma, we get an exact sequence
$$
0 \la K/Z^m(X_\alpha) \la (X^m_\alpha+W)/X^m_\alpha \la Z^{m+1}_{\alpha+1}/Z^{m+1}(X_\alpha) \la 0. $$
Since $Z^{m+1}_{\alpha+1}/Z^{m+1}(X_\alpha)$ is $<\kappa$-presentable and $(X^m_\alpha+W)/X^m_\alpha \cong W/(X^m_\alpha \cap W)$ is $<\kappa$-generated, the factor $K/Z^m(X_\alpha)$ is $<\kappa$-generated by Lemma~\ref{lem:size of kernels}. Hence there is $W' \subseteq Z^m(X_\alpha)$ which is $<\kappa$-generated and such that $K = Z^m(X_\alpha) + W'$. By Theorem~\ref{thm:generalized hill lemma}(H4), there is a set $S \in \clL_m$ of cardinality $<\kappa$ \st $W' \subseteq \ell_m(S)$. Now we set $Z^m_{\alpha+1} = Z^m(X_\alpha) + \ell_m(S)$ and $X^m_{\alpha+1} = X^m_{\alpha} + W + \ell_m(S)$. Observe that using modularity of the lattice of subobjects of $X^m$, we have
$$
X^m_{\alpha+1} \cap Z^m(X) = K + \ell_m(S) = Z^m(X_\alpha) + \ell_m(S) = Z^m_{\alpha+1}.
$$
Further, using Theorem~\ref{thm:generalized hill lemma} and Corollary~\ref{cor:small objects closed under small filtrations}, one readily infers that $Z^m_{\alpha+1} = Z^m(X_\alpha) + \ell_m(S) \in \clH_m$ and $Z^m_{\alpha+1}/Z^m(X_\alpha)$ is $\clS$-filtered and $<\kappa$-presentable. Finally, since $\ell_m(S) \subseteq Z^m(X)$, we have $\Img (d^m_X \restriction X^m_{\alpha+1}) = Z^{m+1}_{\alpha+1}$. Therefore, $X^m_{\alpha+1} \subseteq X^m$ has all the required properties. This finishes the induction step for $m$ and also the construction.
\end{proof}

To finish the proof of Theorem~\ref{thm:deconstr_cpxs}, we focus on part (3). Our exposition here uses ideas of Gillespie from~\cite[\S3]{G3}, which were also nicely presented by Hovey in~\cite[\S7]{H3}.

\begin{prop} \label{prop:deconstr_dgF}
Let $\G$ be a Grothendieck category, and $\F \subseteq \G$ be a class of objects which is generating and deconstructible. Then each $X \in \dg\F$ is a summand of an object filtered by stalk complexes of the form $F[n]$ with $F \in \F$ and $n \in \mathbb{Z}$. In particular, $\dg\F$ is deconstructible.
\end{prop}

\begin{proof}
Let us fix a set $\clS \subseteq \F$ which is generating in $\G$ and \st $\F = \Filt\clS$, and denote
$$ \C = \{C \in \G \mid \ExtG(S,C) = 0 \textrm{ for each } S \in \clS\}. $$
Let $\Q$ be the set of all stalk complexes of the form $S[n]$ for an integer $n$ and $S \in \clS$, and denote by $\F'$ the class of all direct summands of objects from $\F$. With this notation (and using Notation~\ref{not:F tilde and dg-F}), we will prove that
\begin{enumerate}
 \item $\tilde\C = \{Y \in \Cpx\G \mid \Ext^1_{\Cpx\G}(Q,Y) = 0 \textrm{ for each } Q \in \Q\}$, and
 \item $\dg{\F'} = \{X \in \Cpx\G \mid \Ext^1_{\Cpx\G}(X,Y) = 0 \textrm{ for each } Y \in \tilde\C\}$.
\end{enumerate}
The first statement of the proposition will then become an immediate consequence of Proposition~\ref{prop:cotorsion_and_filt}. Indeed, note that $\Filt\Q$ is a generating class in $\Cpx\G$ since it contains complexes of the form
$$
\qquad \dots \la 0 \la 0 \la S \overset{1_S}\la S \la 0 \la 0 \la \dots
$$
with $S \in \clS$, and $\clS$ is assumed to be generating in $\G$. Then $\dg{\F'}$ is the closure of $\Filt\Q$ under direct summands by Lemmas~\ref{lem:complexes over G} and~\ref{lem:deconstr_filt}, and $\dg\F = \Cpx\F \cap \dg{\F'}$ directly from the definition.

Let us prove (1) and (2). For (1), assume first that $Y \in \Cpx\G$ is \st $\Ext^1_{\Cpx\G}(\Q,Y) = 0$. To begin with, we show that $Y$ is acyclic. Suppose it is not, so $H^n(Y) \ne 0$ for some $n \in \Z$. Then there is a chain complex morphism $f: S[-n] \to Y$ with $S \in \mathcal{S}$, which is not null-homotopic. Indeed, there is an epimorphism in $\G$ of the from $\bigoplus_{i \in I} S_i \to Z^n(Y)$ with $S_i \in \mathcal{S}$ for each $i \in I$. If $H^n(Y) \ne 0$, at least one of the components $S_i \to Z^n(Y)$ cannot factor through the differential map $Y^{n-1} \to Z^n(Y)$ of $Y$, and we can take $S = S_i$ and the induced chain complex morphism $S[-n] \to Y$ for $f$. Having such $f$ and using Lemma~\ref{lem:stable_hom}, the mapping cone $C_f$ of $f$ fits into a componentwise split but non-split exact sequence
$$ 0 \la Y \la C_f \la S[-n+1] \la 0, $$
which contradicts $\Ext^1_{\Cpx\G}(\Q,Y) = 0$. Hence $X$ is acyclic.

Next we show that $Z^n(Y) \in \C$ for all $n \in \Z$. If not, then we would have a non-split extension $0 \to Z^n(Y) \to E \to S \to 0$ in $\G$ for some integer $n$ and $S \in \clS$. This would induce, using the pushout diagram
$$
\begin{CD}
0 @>>> Z^n(Y) @>>> Y^n @>>> Z^{n+1}(Y) @>>> 0\phantom{,}    \\
@.     @VVV       @VVV          @|                          \\
0 @>>>   E    @>>> W^n @>>> Z^{n+1}(Y) @>>> 0,
\end{CD}
$$
a non-split extension
$$ 0 \la Y \la W \la S[-n] \la 0 $$
in $\Cpx\G$, which is a contradiction again. Hence $Z^n(Y) \in \C$ for each integer $n$ and consequently $Y \in \tilde\C$.

Assume on the other hand that $Y \in \tilde\C$. Since $\C$ is extension closed, any extension of the form
$$ 0 \la Y \la W \la S[n] \la 0 $$
with $S \in \clS$ is componentwise split. In particular, we have the equality
$$ \Ext^1_{\Cpx\G}(S[n],Y) = \Ext^1_\CGcs(S[n],Y). $$
To prove that the $\Ext$-groups above vanish, we must in view of Lemma~\ref{lem:stable_hom} show that any morphism $f\colon S[n-1] \to Y$ is null-homotopic. To see this, note that such $f$ is given by a morphism $f': S \to Z^{n-1}(Y)$ in $\G$, and $f'$ factors through the epimorphism $Y^{n-2} \to Z^{n-1}(Y)$ since $Z^{n-2}(Y) \in \C$. It follows that $\Ext^1_{\Cpx\G}(\Q,Y) = 0$ and the proof of (1) is finished.

For (2), first assume that $X$ is a complex over $\G$ \st $\Ext^1_{\Cpx\G}(X,\tilde\C) = 0$. We claim that $X$ must belong to $\Cpx{\F'}$. Equivalently by Proposition~\ref{prop:cotorsion_and_filt}, we must show that $\ExtG(X^n,\C) = 0$ for each component $X^n$ of $X$. To this end, we employ a similar argument as in \cite[Proposition 3.6]{G3}. Namely, note that if there is a non-split extension $0 \to C' \to V^n \to X^n \to 0$ for an integer $n$ and $C' \in \C$, then the pullback diagram
$$
\begin{CD}
0 @>>>  C' @>>> V^{n-1} @>>> X^{n-1} @>>> 0    \\
@.      @|       @VVV          @VVV            \\
0 @>>>  C' @>>>   V^n   @>>>   X^n   @>>> 0
\end{CD}
$$
induces a non-split extension
$$ 0 \la C \la V \la X \la 0 $$
in $\Cpx\G$, where $C \in \tilde\C$ is the correspondingly shifted complex of the form
$$
\qquad \dots \la 0 \la 0 \la C' \overset{1_{C'}}\la C' \la 0 \la 0 \la \dots.
$$
This would contradict our assumption, so the claim is proved. Next, note that if $C \in \tilde\C$, then any extension of $C$ by $X$ must be componentwise split, so
$$ \Ext^1_{\Cpx\G}(X,C) = \Ext^1_{\CGcs}(X,C). $$
Hence, $X \in \dg{\F'}$ directly by Lemma~\ref{lem:stable_hom} and the definition; see Notation~\ref{not:F tilde and dg-F}.

If on the other hand $X \in \dg{\F'}$, we just retrace the steps. Namely, we observe that for each $C \in \C$ we have $\Ext^1_{\CGcs}(X,C) = 0$ because of Lemma~\ref{lem:stable_hom} and $\Ext^1_{\Cpx\G}(X,C) = \Ext^1_{\CGcs}(X,C)$ because $X \in \Cpx{\F'}$. Hence $\Ext^1_{\Cpx\G}(X,\C) = 0$ and the proof of (2) is complete.

Finally, to prove the deconstructibility of $\dg\F$, note that we have proved that $\dg{\F'}$ consists precisely of direct summands of complexes from $\Filt\Q$. Thus, both $\dg{\F'}$ and $\dg\F = \Cpx\F \cap \dg{\F'}$ are deconstructible by Propositions~\ref{prop:add_deconstr} and~\ref{prop:deconstr_C(F)}.
\end{proof}

\begin{proof}[Proof of Theorem~\ref{thm:deconstr_cpxs}]
The theorem now follows directly by Propositions~\ref{prop:deconstr_C(F)}--\ref{prop:deconstr_dgF}, when taking $\kappa$ appropriately large.
\end{proof}

% ------------------------------------------------------------------------------
% Appendices
\appendix
\section{Properties of Grothendieck categories}
\label{sec:properties of Grothendieck categories}

In the text we need some basic properties of $<\kappa$-presentable and $<\kappa$-generated objects in Grothendieck categories. We also need to know that the lattice of subobjects of an object in a Grothendieck category is always modular and upper continuous. These facts can be mostly found in the monographs~\cite{S} by Stenstr{\"o}m and~\cite{GaU} by Gabriel and Ulmer. Here we give a short overview of the necessary properties together with appropriate references or short proofs. We start with a basic property of any Grothendieck category.

\begin{lem} \label{lem:locally presentable}
Any Grothendieck category $\G$ is locally $<\kappa$-presentable for some infinite regular cardinal $\kappa$. In particular, for any $X \in \G$ there is an infinite regular cardinal $\lambda = \lambda(X)$ \st $X$ is $<\lambda$-presentable.
\end{lem}

\begin{proof}
This is a well-known consequence of the Popescu-Gabriel Theorem~\cite[X.4.1]{S}. Namely, if $G \in \G$ is a generator and $R = \End_\G(G)$, then $H = \Hom_\G(G,-): \G \to \Mod{R}$ induces an equivalence of $\G$ with a full subcategory $\G' \subseteq \Mod{R}$. By its definition on~\cite[pp. 198--199]{S}, $\G'$ is closed under $\kappa$-direct limits in $\Mod{R}$ for some infinite regular cardinal $\kappa$ and $H(G) = R$ is clearly $<\kappa$ presentable in $\G'$.
\end{proof}

Next we have an important characterization of $<\lambda$-presentable objects for $\lambda$ large enough.

\begin{lem} \label{lem:size of kernels}
Let $\G$ be a Grothendieck category and $\kappa$ an infinite regular cardinal \st $\G$ is locally $<\kappa$-presentable. Then the following are equivalent for an object $X \in \G$ and a regular cardinal $\lambda\ge\kappa$:
\begin{enumerate}
\item $X$ is $<\lambda$-presentable.
\item $X$ is $<\lambda$-generated and, whenever $0 \to K \to E \to X \to 0$ is a short exact sequence in $\G$ \st $E$ is $<\lambda$-generated, $K$ is also $<\lambda$-generated.
\end{enumerate}
\end{lem}

\begin{proof}
A straightforward generalization of the proof for~\cite[Proposition V.3.4]{S} (which is given for $\kappa=\aleph_0$) applies. See also~\cite[6.6(e)]{GaU}.
\end{proof}

The previous proposition provides us with a way to recognize $<\lambda$-generated and presentable objects using their presentations by a fixed family of $<\kappa$-presentable generators.

\begin{lem} \label{lem:identifying lambda presented} \cite[7.6 and 9.3]{GaU}
Let $\G$ be a Grothendieck category, $\kappa$ be an infinite regular cardinal, and assume $\G$ has a generating set $\clS$ consisting of $<\kappa$-presentable objects. Then the following hold for an object $X \in \G$ and a regular cardinal $\lambda\ge\kappa$:
\begin{enumerate}
\item $X$ is $<\lambda$-generated \iff there exists an exact sequence
$$ \bigoplus_{i \in I} S_i \la X \la 0 $$
with $\card{I}<\lambda$ and $S_i \in \clS$ for all $i \in I$.

\item $X$ is $<\lambda$-presentable \iff there exists an exact sequence
$$ \bigoplus_{j \in J} S_j \la \bigoplus_{i \in I} S_i \la X \la 0 $$
with $\card{I}, \card{J}<\lambda$ and $S_i, S_j \in \clS$ for all $i \in I$ and $j \in J$.
\end{enumerate}
\end{lem}

\begin{proof}
The `if' part of (1) follows from the fact that $<\lambda$-generated objects are closed under factors and coproducts with $<\lambda$ summands. For the `only if' part of (1), take an epimorphism $p: \bigoplus_{u \in U} S_u \to X$ for some set $U$ and denote by $\I$ be the set of all subsets of $U$ of cardinality $<\lambda$. Then $X = \li_{I \in \I} \Img (p \restriction \bigoplus_{i \in I} S_i)$ is a $\lambda$-directed limit of monomorphisms, so $\Img (p \restriction \bigoplus_{i \in I} S_i) = X$ for some $I \in \I$. Part (2) is an easy consequence of (1) and Lemma~\ref{lem:size of kernels}.
\end{proof}

It is well-known for module categories that the classes of $<\lambda$-generated and $<\lambda$-presented modules are closed under extensions. We have an analogue for Grothendieck categories.

\begin{lem} \label{lem:small objects closed under ext}
Let $\G$ be a Grothendieck category and $\kappa$ an infinite regular cardinal \st $\G$ is locally $<\kappa$-presentable. Then for any regular cardinal $\lambda\ge\kappa$, the classes of $<\lambda$-generated and $<\lambda$-presentable objects are closed under extensions.
\end{lem}

\begin{proof}
For the class of $<\lambda$-generated objects, an obvious generalization of~\cite[Lemma V.3.1(ii)]{S} applies. We do not even use the assumption that $\G$ is locally $<\kappa$-presentable.

Suppose we have a short exact sequence $0 \to X \to Y \overset{p}\to Z \to 0$ \st $X,Z$ are $<\lambda$-presentable. Then $Y$ is $<\lambda$-generated. In order to apply Lemma~\ref{lem:size of kernels}, assume that $0 \to K \to E \overset{q}\to Y \to 0$ is any short exact sequence \st $E$ is $<\lambda$-generated and consider the commutative diagram with exact rows, where $L = \Ker (pq)$:
$$
\begin{CD}
  0 @>>> L @>>> E @>>>  Z @>>> 0  \\
  @.  @VtVV   @VqVV     @|        \\
  0 @>>> X @>>> Y @>p>> Z @>>> 0
\end{CD}
$$
Since $Z$ is $<\lambda$-presentable, $L$ is $<\lambda$-generated by Lemma~\ref{lem:size of kernels}. Since $X$ is $<\lambda$-presentable, $K \cong \Ker t$ is $<\lambda$-generated. Applying Lemma~\ref{lem:size of kernels} for the third time, it follows that $Y$ is $<\lambda$-presentable.
\end{proof}

In the proof of the Generalized Hill Lemma (Theorem~\ref{thm:generalized hill lemma}), the following consequence is used:

\begin{cor} \label{cor:small objects closed under small filtrations}
Let $\G$ be a locally $<\kappa$-presentable Grothendieck category for some infinite regular cardinal $\kappa$. If $X$ is an object with a filtration $(X_\alpha \mid \alpha\le\sigma)$ by $<\kappa$-presentable objects and \st $\sigma<\kappa$, then $X$ is $<\kappa$-presentable.
\end{cor}

\begin{proof}
This is easily proved by induction on $\sigma$ using Lemma~\ref{lem:small objects closed under ext} and the fact that the class of $<\kappa$-presentable objects is closed under taking colimits of diagrams of size less than $\kappa$ (see~\cite[Satz 6.2]{GaU}).
\end{proof}

Finally, we establish some properties of lattices of subobjects. Recall that if $(\clL, \vee, \wedge)$ is a complete lattice, we call $\clL$ \emph{upper continuous} (cf.~\cite[\S III.5]{S}) if
$$ \big( \bigvee_{d \in D} d \big) \wedge a = \bigvee_{d \in D} (d \wedge a) $$
whenever $a \in \clL$ and $D \subseteq \clL$ is a directed subset. Now we have:

\begin{lem} \label{lem:upper cont and modular}
Let $\G$ be a Grothendieck category. If $X \in \G$ is an object, the lattice $(\Subobj(X), \Sigma, \cap)$ of subobjects is a complete modular upper continuous lattice.
\end{lem}

\begin{proof}
This is proved in~\cite[Propositions IV.5.3 and V.1.1(c)]{S}.
\end{proof}

% ------------------------------------------------------------------------------
\bibliographystyle{abbrv}
\bibliography{deconstr_bib}

\end{document}